\documentclass
{article}
\usepackage[titletoc,toc]{appendix}
\usepackage{amsmath,amssymb,amsthm,color,graphicx,setspace,verbatim}
\usepackage{hyperref}

\usepackage{algorithmic}

\newcommand{\Red}[1]{\color{red} #1}
\newcommand{\Green}[1]{\color{green} #1}
\newcommand{\Blue}[1]{\color{blue} #1}

\newcommand{\dgeq}{\succeq}
\newcommand{\fm}{{\mathfrak m}}
\newcommand{\gr}{{\mathfrak{gr}}}
\newcommand{\p}{{\partial}}

\DeclareMathOperator{\ord}{ord}
\DeclareMathOperator{\Spec}{Spec}
\DeclareMathOperator{\Dual}{Dual}
\DeclareMathOperator{\POStoN}{(\bZ_{\geq 0})^N}

\DeclareMathOperator{\HP}{HP} 

\DeclareMathOperator{\initial}{in}

\newtheorem{theorem}{Theorem}[section]
\newtheorem{lemma}[theorem]{Lemma}

\newtheorem{example}[theorem]{Example}
\newtheorem{proposition}[theorem]{Proposition}
\newtheorem{problem}[theorem]{Problem}
\newtheorem{algorithm}[theorem]{Algorithm}
\newtheorem{remark}[theorem]{Remark}

\newcommand{\codim}{\operatorname{codim}}

\newcommand{\TruncatedTruncation}{\operatorname{DoubleTruncation}}

\newcommand{\IsWitnessPolynomial}{\operatorname{IsWitnessPolynomial}}
\newcommand{\alg}[1]{\operatorname{#1}}

\newcommand{\bN}{{\mathbb N}}

\newcommand{\bZ}{{\mathbb Z}}

\newcommand{\bC}{{\mathbb C}}

\newcommand{\bQ}{{\mathbb Q}}

\newcommand{\bV}{{\mathbb V}}

\newcommand\sA{{\mathcal A}}

\newcommand{\Span}{\operatorname{span}}

\newcommand\Var{{\bV}}
\def\gcorner{g-corner}

\newcommand{\ideal}[1]{\langle #1 \rangle}
\newcommand{\LT}{\operatorname{in}_\geq}
\newcommand{\LTg}{\operatorname{in}_\succeq}

  \def \tab#1{\kern #1 truein}

\def \dId {I^{(d)}}
\def \dXd {X^{(d)}}
\def \dYd {Y^{(d)}}

\newcommand\Ass{\operatorname{Ass}}
\newcommand\Tor{\operatorname{Tor}}
\newcommand\calC{{\mathcal C}} %
\newcommand\calO{{\mathbf O}} 
\newcommand\calN{{\mathcal N}} 

\newcommand{\dehomog}{\varphi}

\begin{document}
\title{Numerical algorithms for\\ detecting embedded components}
\author{
Robert Krone\thanks{School of Mathematics, Georgia Tech, Atlanta GA, USA ({\tt rkrone3@math.gatech.edu}). Partially supported by NSF grant DMS-1151297}
\and
Anton Leykin\thanks{School of Mathematics, Georgia Tech, Atlanta GA, USA ({\tt leykin@math.gatech.edu}). Partially supported by NSF grants DMS-0914802 and DMS-1151297}}
\maketitle


\begin{abstract}
We produce algorithms to detect whether a complex affine variety computed and presented numerically by the machinery of numerical algebraic geometry corresponds to an associated component of a polynomial ideal.



\end{abstract}


\section{Introduction}\label{Sec:intro}
An algorithmic approach to complex algebraic geometry known as {\em numerical algebraic geometry} (numerical AG, see~\cite{SVW9,Sommese-Wampler-book-05}) provides fast approximate methods to {\em solve} systems of polynomial equations. 
In the case when the solution set is a finite set of points {\em polynomial homotopy continuation} techniques are able to find approximations to all solutions. 
In the case when the solution set is positive-dimensional, it is a union of irreducible complex affine varieties and {\em numerical irreducible decomposition}~\cite{SVW1} is performed to capture the information about the irreducible pieces with numerical data stored in the so-called {\em witness sets}.
In ideal-theoretic terms, given a generating set of an ideal $I$ in the polynomial ring $R = \bC[x] = \bC[x_1,\dots,x_N]$, the numerical irreducible decomposition gives a numerical description of the components corresponding to the prime ideals $P_i$ in the decomposition of the radical $\sqrt I = P_1\cap\cdots\cap P_r$.

The goal of {\em numerical primary decomposition}~\cite{Leykin:NPD} is to find a generic point on every component of the affine scheme $\Spec(R/I)$; in ideal-theoretic terms, find a generic\footnote{Here and throughout the paper we say a ``generic point on component'' to refer to a point in the complement of a proper Zariski closed subset of the component containing the ``degeneracy locus'' dictated by the context. One can trust numerical methods mentioned so far to produce random points on components that avoid the degeneracy locus ``with probability 1''.}  point on the component $\Var(P)$ for every associated prime ideal $P\in\Ass(R/I)$.  In general a primary decomposition will include {\em embedded components} not found in an irreducible decomposition, whose corresponding primes strictly contain other associated primes of $I$.


There are various methods that produce generic points on pieces of the singular locus that may or may not represent embedded components. We refer to such pieces as {\em suspect components};  if a suspect component does not turn out to be embedded, we call it a {\em pseudocomponent}. 

We see answering the following question algorithmically as one of the first stepping stones to extending numerical AG methods to the full generality of affine schemes.
\begin{problem}[Main Problem] For
\begin{enumerate}
  \item an ideal $I\subset R$ given by a finite generating set,
  \item a point $y\in\bC^n$, and
  \item generic points $y_1,\ldots,y_r$ on a collection of components $\Var(P_1),\ldots,\Var(P_r)$, $P_i\in\Ass(R/I)$, that contain $y$,
\end{enumerate}
decide whether there is a component $\Var(P)$, $P\in\Ass(R/I)$, that contains $y$ and is distinct from $\Var(P_i)$, for $i=1,\ldots, r$.
\end{problem}

We assume that part (1) of the input is {\em exact}, i.e., the coefficients of the generators of the ideal $I$ are known {\em exactly}. However, following the framework of numerical AG we assume no access to generators of prime ideals in the parts (2) and (3), nor to exact values for $y,y_1,\ldots,y_r$.

We shall describe ideals of a polynomial ring $R$ as well as the ideals of the localization $R_y$ of $R$ at a point $y\in\bC^N$
in terms of the {\em Macaulay dual spaces}.

For convenience, hypothetically, the reader may assume also that the points $y, y_1,\ldots,y_r$ in the parts (2) and (3) are exact and the Macaulay dual space algorithm is exact. With this assumption our algorithms become {\em purely symbolic}. In reality, our approach is hybrid: we state what numerical ingredients are necessary in~\S\ref{sec:oracles-and-conclusion}. 

\smallskip

The algorithms in this article are implemented in {\em Macaulay2}~\cite{M2www} with parts of code residing in the packages {\em NumericalHilbert}~\cite{Krone:NumericalHilbert} and {\em NumericalAlgebraicGeometry}~\cite{Leykin:NAG,NAGwww}. Instructions on steps necessary to reproduce results for the examples are posted at
\begin{center}\url{www.rckr.one/embedded-component-test/}.\end{center}

\medskip

The beginning of \S\ref{Sec:Prelim} mostly covers basic preliminaries: Macaulay dual spaces and their connection to local polynomial rings, (local) Hilbert function, regularity index, s- and g- corners. Also \S\ref{Sec:Prelim} reviews the operation of taking a colon ideal through the numerical lens and develops the local ideal membership test. The numerical primary decomposition is revisited in \S\ref{Sec:NPD}; this section is not essential, but is used in setting up examples and to provide a better understanding of the general context. 
The main part of this work, \S\ref{Sec:embedded-test}, develops algorithms for embedded component testing. One important side result worth highlighting is Theorem~\ref{thm:Ass-of-a-generic-slice}. It concerns associated components of the generic hyperplane section of an affine scheme and makes the dimension reduction possible in our approach.
Discussion of numerical ingredients and conclusion is in \S\ref{sec:oracles-and-conclusion}.

\medskip

\noindent {\bf Acknowledgments.} We are grateful to Jon Hauenstein for helpful discussions that started at the Institut Mittag-Leffler, which kindly hosted both Hauenstein and Leykin in the Spring of 2011. 

We also would like to thank numerous people, in particular, Uli Walther and Karl Schwede, for discussions of matters in~\S\ref{subsec:positive-dim-suspect} as well as to Hailong Dao and Sasha Anan'in who produced a proof of Lemma~\ref{lem:colon-hyperplane-commute} (via {\tt mathoverflow.net}). 

\section{Preliminaries}\label{Sec:Prelim}

For $\alpha \in \POStoN$ and $y\in\bC^N$, let
\begin{align*}
x^\alpha &= x_1^{\alpha_1}\cdots x_N^{\alpha_N}\,,\\
|\alpha|&= \sum_{i=1}^N \alpha_i\,,\\
\alpha! &= \alpha_1!\alpha_2!\dots\alpha_N!\,,\\
\p^\alpha &= \frac{1}{\alpha!}\frac{\p^{|\alpha|}}{\p x^\alpha}\,,
\end{align*}
and the map $\p^\alpha[y]: R \rightarrow \bC$~be defined by~$\p^\alpha[y](g) = (\p^\alpha g)(y).$

Instead of $\partial^\alpha[y]$ we sometimes write $\partial^{x^\alpha}[y]$, for example, $\partial^1 - \partial^{y} + \partial^{x^2yz}$, and
when the point $y$ is implied $\partial^\alpha[y]$ we write~$\partial^\alpha$.
For $y \in \bC^N$, let \[D_y = \Span_{\bC}\left\{\partial^\alpha[y]~|~\alpha \in \POStoN \right\}\]
be the vector space of differential functionals at $y$. This linear space is graded by the {\em order}, for a finite sum $q = \sum c_\alpha\p^\alpha$,
\[
\ord q = \max_{c_\alpha\neq 0} |\alpha|.
\]
The {\em homogeneous} part of order $i$ of $q\in D_y$ is referred to as $q_i$. This grading is the associated graded linear space of the filtration $D_y^*$:
\[D_y^0 \subset D_y^1 \subset D_y^2 \subset \ldots \text{, where }D_y^i = \{q\in D_y~|~\ord q\leq i\}\}.\]

The {\it Macaulay dual space}, or simply {\em dual space},
is the $\bC$-space of differential functionals that vanish
at $y$ for an ideal $I\subset\bC[x]=\bC[x_1,\dots,x_N]$~is
\begin{equation}\label{Eq:Dual}
D_y[I] = \{q\in D_y~|~q(g)=0\hbox{~for all~}g\in I\}.
\end{equation}
The dual space $D_y[I]$ is a linear subspace of $D_y$, a basis of $D_y[I]$ is called a {\it dual basis} for $I$.

\subsection{Duality}\label{Sec:Local-vs-Dual}
To help the reader, in this section we list key facts about ideal and dual space correspondence; see~\cite{KL:eliminating-dual-spaces} for proofs and references.

Without loss of generality, we may assume $y=0\in\bC^N$.
Consider the local ring $R_0 = R_\fm$ where $\fm = \ideal{x_1,\ldots,x_N}$. Let the space of dual functionals be defined as above replacing $R$ (polynomial) with $R_0$ (rational functions with denominators not vanishing at 0).

\begin{remark}\label{rmk:local-dual}
Ideals in $R$ with all primary components containing the origin are in one-to-one correspondence with ideals in the local ring $R_0$ given by extension ($I\subset R$ extends to $IR_0\subset R_0$) and contraction  ($I\subset R_0$ contracts to $I\cap R \subset R$, all of whose primary components contain the origin).

For ideal $I \subset R$, the dual space $D_0[I]$ is identical to the dual space of its extension in $R_0$, $D_0[IR_0]$. Note that, for $I \subset R$, $f \in IR_0 \cap R$ if and only if $q(f) = 0$ for all $q \in D_0[I]$.
\end{remark}

It follows from the remark that for  $J_1,J_2 \subset R_0$, we have $J_1 \subsetneq J_2$ if and only if $D_0[J_1] \supsetneq D_0[J_2]$. Hence, an ideal $J \subset R_0$ is uniquely determined by its dual space $D_0[J]$.

$R$ naturally acts on the the dual space by {\em differentiation}.
\begin{eqnarray*}
  x_i:  D_y &\to& D_y\\
             \p^\alpha &\mapsto& \p^{\alpha - e_i}, \ \ \ \ (i=1,\ldots,N),
\end{eqnarray*}
where $\p^\beta$ is taken to be $0$ when any entry of $\beta$ is less than zero.
For all $q \in D_0$ and $f \in R_0$, note that $(x_i \cdot q)(f) = q(x_i f)$, so the action of $x_i$ on a functional can be seen as pre-multiplication by $x_i$.

A subspace $L \subset D_0$ is the dual space of some ideal \mbox{$I_L \subset R_0$} if and only if it is closed under differentiation: $x_i \cdot L \subset L$ for all \mbox{$0 \leq i \leq N$}.

The map $$\Dual : \{\text{ideals of } R_0\}  \to \{\text{subspaces of }D_0\text{ closed under differentiation}\}$$ defined by $\Dual(J) = D_0[J]$ is a bijection and provides another way to characterize the dual space.

An alternative characterization of the dual space can be given via the following Proposition. 
\begin{proposition}\label{Prop:closedness}
For ideal $J=\langle f_1,\dots,f_n\rangle\subset R_0$, let $L$ be the maximal subspace of $D_0$ that is closed under differentiation and satisfies $q(f_i) = 0$ for all $q \in L$ and each $0 \leq i \leq n$.  Then $L = D_0[J]$.
\end{proposition}

From Proposition~\ref{Prop:closedness} it follows that for $I = \ideal{f_1,\dots,f_n}$, a dual element $q$ is in $D_0[I]$ if and only if $q(f_i) = 0$ and $x_j \cdot q \in D_0[J]$ for each $0 \leq i \leq n$ and $0 \leq j \leq N$.  Note that this leads to a completion algorithm for computing~$D_y^k[I]$ (see, e.g.,~ \cite{Mourrain:inverse-systems}) assuming $y$ is in the vanishing set of $I$:
\begin{algorithmic}
\STATE $D_y^0[I] \leftarrow \Span_\bC(\p^0)$
\FOR{$i = 1 \to k$}
\STATE $D_y^i[I] \leftarrow \{ q\in D_y~|~x_j \cdot q \in D_y^{i-1}[I] \mbox{ for all } j=1,\ldots,N \mbox{ and } q(f_i) = 0 \mbox{ for all } i=1,\ldots,n \}$
\ENDFOR
\end{algorithmic}

For ideals $J_1,J_2 \subset R_0$, the following can be readily shown:
\begin{align*}
 D_0[J_1 + J_2] &= D_0[J_1] \cap D_0[J_2]\,,\\
 D_0[J_1 \cap J_2] &= D_0[J_1] + D_0[J_2].
\end{align*}

For the truncated dual spaces the second equality holds if $J_1$ and $J_2$ are homogeneous ideals. In general, we have only one inclusion:
  $$D_0^k[J_1\cap J_2] \supset D_0^k[J_1]+D_0^k[J_2].$$
Since $D_0^k[J_1\cap J_2]$ is finite dimensional, it follows that
  $$D_0^k[J_1\cap J_2] \subset D_0^l[J_1]+D_0^l[J_2]$$
for some $l$.

\subsection{Primal and dual monomial order}
Let $\geq$ be a local monomial ordering (a total order on the monomials which respects multiplication and has $1$ as the largest monomial), which we shall refer to as a {\em primal order}. For
$g = \sum_{\alpha} a_\alpha x^\alpha$, a nonzero polynomial,
the {\em initial term} with respect to $\geq$ is
the largest monomial with respect to $\geq$ that has a nonzero coefficient, namely
$$\initial_{\geq}(g) = \max_\geq\{x^\alpha~|~a_\alpha\neq 0\}.$$
For an ideal $I$, the {\it initial terms} of $I$ with respect to $\geq$
is the set of initial terms with respect to $\geq$ of all the elements of $I$, namely
$$\initial_{\geq}(I) = \{\initial_{\geq}(f)~|~f \in I\}.$$
A monomial is called a {\it standard monomial} of $I$
with respect to $\geq$ if it is not a member of $\initial_\geq(I)$.

We shall order the monomial differential functionals via the {\em dual order}:
$$
\partial^\alpha \dgeq \partial^\beta\ \Leftrightarrow\ x^\alpha \leq x^\beta,
$$
the order opposite to $\geq$.

The {\em initial term} $\initial_\dgeq(q)$ of $q$
is the largest monomial differential functional that has a nonzero coefficient.

A dual basis that has distinct initial terms is called a {\em reduced dual basis}.
Using a (possibly infinite dimensional) Gaussian elimination procedure,
it is easy to see that any dual basis can be transformed into a reduced dual basis.

\begin{theorem}[Theorem 3.3 of~\cite{KL:eliminating-dual-spaces}]\label{theorem:complementary-staircases}
For an ideal $I \subset R$ the monomial lattice $\bN^N$ is a disjoint union of $\LTg D_0[I]$ and $\LT I$.
\end{theorem}

\subsection{Local Hilbert function, g- and s-corners}\label{Sec:Hilbert}

The Hilbert function of an ideal $I \subset R_0$ provides combinatorial information about $I$ that can be computed numerically using truncated dual spaces.  

For an ideal $I\subset R_0$ define the {\em Hilbert function} as
\begin{align*}
H_I(k) &= \dim_\bC (\gr(R_0/I)_k) = \dim_\bC\left(\frac{I+\fm^{k+1}}{I+\fm^{k}}\right) \\
&= \dim_\bC\left(R_0/(I+\fm^{k+1})\right) - \dim_\bC\left(R_0/(I+\fm^{k})\right).
\end{align*}

The Hilbert function is determined by the initial ideal with respect to the primal monomial order (that respects the degree): $H_I(k) = H_{\LT(I\cap R)}(k),\ \text{for all }k\in\bN.$

We can compute the Hilbert function using truncated dual spaces:
$$H_I(k) = \dim_\bC D_0^k[I] - \dim_\bC D_0^{k-1}[I],\ \text{for }k\geq 0, $$
where $\dim_\bC D_0^{-1}[I]$ is taken to be 0.

The Hilbert function $H_I(k)$ is a polynomial for all $k \geq m$ for a sufficiently large $m\geq 0$ (see, e.g., \cite[Lemma 5.5.1]{Singular-book-02}). This polynomial is called the {\em Hilbert polynomial} $\HP_I(k)$. If the dimension of $I \subset R_0$ is $d$, then $\HP_I(k)$ is a polynomial of degree $d-1$.  

The {\em regularity index} of the Hilbert function is
$$\rho_0(I) = \min\{\,m \,: \,H_I(k)=\HP_I(k)\text{ for all } k \geq m\,\}.$$

For a 0-dimensional ideal $I$, the {\em multiplicity} $\mu_0(I)$ is defined as $\dim_\bC (R_0/I) = \dim D_0[I]$.  For $I$ of dimension $d > 0$ with Hilbert polynomial $\HP_I(k) = a_{d-1}k^{d-1} + O(k^{d-2})$ the multiplicity is defined as
 \[ \mu_0(I) = a_{d-1}(d-1)!. \]

The multiplicity of $I$ can be interpreted geometrically as follows.  For $I \subset R_0$ with dimension $d$, let $L \subset R$ be a generic affine plane of codimension $d$.  Then $J = (I \cap R) + L$ is a 0-dimensional ideal and the points of $\bV(J)$ are smooth points of $\bV(I \cap R)$.  The multiplicity of $I$ is the same as that of $J$, which is the sum of the local multiplicities of the points in $\bV(J)$.
In particular, this means that the multiplicity of $I$ can be computed numerically: the points $\bV(J)$ approximated by homotopy continuation and then the local multiplicities at these points obtained via dual spaces.

\smallskip 

We refer to the minimal monomial generators of a monomial ideal $M$ as {\em \gcorner{}s}. We call a monomial $x^\alpha$ an {\em s-corner} of $M$ when $x_i x^\alpha \in M$ for all $i=1,\ldots,n$.  For a general ideal $I$, the g-corners and s-corners of $I$ will refer to the g-corners and s-corners of the monomial ideal $\LT I$, respectively.\footnote{{\bf g-} and {\bf s-} stand for {\bf generators} of $\LT I$ and monomials spanning the {\bf socle} of the quotient $R_0/\LT I$, respectively.}

\begin{figure}[ht]\label{fig:stair}
  \centering
  \includegraphics[width=.7\columnwidth]{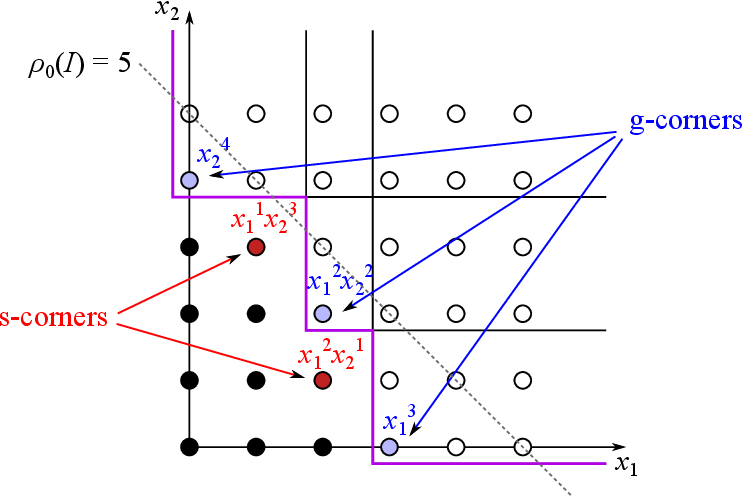}
  \caption{The ``staircase'' of monomial ideal $I = \ideal{x_1^3,x_1^2x_2^2,x_2^4}$ in the lattice of monomials.  The regularity index of the Hilbert function is $\rho_0(I)=5$.}
\end{figure}

\begin{remark}
For a 0-dimensional ideal $I$,  The Hilbert regularity index $$\rho_0(I) = \max \{ |\alpha| \,:\, x^\alpha \text{ is an s-corner of }\LT I\} + 1.$$
\end{remark}

Let $F\subset R$ be a finite set of generators of $I$ and let $F^h\subset R[h]$ denote the homogenization of $F$.  Then it is possible to compute $\rho_0(I)$ using the relationship between the truncated dual spaces $D^k_0[I]$  and $D^l_0[\ideal{F^h}]$ together with a stopping criterion for the latter that recovers all \gcorner{}s (and, therefore, all corners) of~$\ideal{F^h}$.  

\begin{remark}\label{rem:HF-algo} Here we outline the idea of the algorithm of \cite{Krone:dual-bases-for-pos-dim} that computes the \gcorner{}s and, therefore, the s-corners and the regularity index.

Let $\varphi:R[h] \to R$ denote the dehomogenization map sending $h$ to 1.  We equip $R[h]$ with the unique graded local order $\geq$ such that for monomials $a,b \in R[h]$ with the same total degree, $a \geq b$ if and only if $\varphi(a) \geq \varphi(b)$.  By calculating a reduced dual basis of $D_0^k[\ideal{F^h}]$ for a given $k$, we find the monomials in the complement of $\LTg D_0^k[\ideal{F^h}]$, which by Theorem~\ref{theorem:complementary-staircases} correspond to the monomials of $\LT\ideal{F^h}$ of degree $\leq k$.  Examining these monomials, we deduce all \gcorner{}s of $\ideal{F^h}$ which have degree $\leq k$.  The calculation is run for successively higher values of $k$ until all of the \gcorner{}s of $\ideal{F^h}$ are found.  

 If $C$ is a set of monomial generators of $\LT \ideal{F^h}$ then $\varphi(C)$ generates $\LT \ideal{F}$.
\end{remark}

\subsection{Quotient ideals and  local ideal membership test}\label{Sec:eliminating}

Recall that any polynomial $g \in R$ defines a differential operator on $D_0$ by $(g\cdot p)(f) = p(gf)$.

\begin{theorem}[Theorem~2.20 of~\cite{KL:eliminating-dual-spaces}]\label{prop:dualOfColonIdeal}
	$D_0[I:\ideal{g}] = g \cdot D_0[I]$.
\end{theorem}

Let $>$ be a primal order on the monomials of the local ring $R_0$, and $\succ$ be the dual order for the dual monomials of $D_0$. For any $p \in D_0$, we must have $\deg \LTg(x_1 \cdot p) \leq \deg \LTg(p) - 1$, since differentiation reduces the degree of each monomial by 1, but may also annihilate the lead term.  Therefore taking the derivative of the dual space truncated at degree $d+1$ we have $x_1 \cdot D_0^{d+1}[I] \subset D_0^{d}[I:\ideal{x_1}]$.  Equality may not hold since there may be some functionals $q \in D_0^{d}[I:\ideal{x_1}]$ with $q = x_1\cdot p$ for some $p \in D_0[I]$ with lead term having degree higher than $d+1$ and is annihilated by $x_1$.  In general, finding $D_0^{d}[I:\ideal{x_1}]$ from the truncated dual space of $I$ may require calculating $D_0^{c}[I]$ up to a very high degree $c$.

Some of these issues can be side-stepped through homogenization.  As in the algorithm described in Remark~\ref{rem:HF-algo}, for $f \in R$, let $f^h \in R[h]$ denote the homogenization of $f$.  Let $\dehomog:R[h] \to R$ be the dehomogenization map, which sends $h$ to 1.

 \begin{proposition}
  $\dehomog(\ideal{F^h} : \ideal{g^h}) = I : \ideal{g}$.
 \end{proposition}

 \begin{proof}
  Suppose $j \in \ideal{F^h} : \ideal{g^h}$, so $jg^h \in \ideal{F^h}$.  Then by dehomogenizing, $\varphi(j)g \in \ideal{F}$ so $\varphi(j) \in I :\ideal{g}$.

  Suppose $j \in I : \ideal{g}$.  Then $jg = \sum_{f \in F} a_f f$ for some $a_f \in R$.  Homogenizing, $h^cj^hg^h = \sum_{f \in F}h^{c_f}a_f^h f^h$ for some non-negative integers $c$ and $c_f$.  Therefore $h^cj^h \in \ideal{F^h} : \ideal{g^h}$ and $\varphi(h^cj^h) = j$.
 \end{proof}

 Since $\ideal{F^h}$ and $g^h$ are both homogeneous,
  \[ g^h\cdot (D_0^d[\ideal{F^h}]) = D_0^{d-e}[\ideal{F^h} : \ideal{g^h}] \]
 where $e$ is the degree of $g^h$.
 
 We will make use of this for a local ideal membership test using the homogenized dual space.  Let $I$ be an ideal of the local ring $R_0$.  If $g$ is not in $I$ then at some degree the Hilbert functions of $I$ and $I + \ideal{g}$ will differ.  We can compute the values of the Hilbert function for successive degrees using the dual space.  If $g$ is in $I$ then $I:\ideal{g} = R_0$.  This can be checked by computing $D_0^d[\ideal{F^h} : \ideal{g^h}]$ for some $d$ and seeing that $h^d$ is in its initial ideal.  This implies that there is some $f \in \ideal{F^h} : \ideal{g^h}$ with $\dehomog(\LT f) = 1$.  Running both tests simultaneously for successive degrees $d$ guarantees termination.
 
\begin{algorithm}\label{alg:IdealMembership} $B = \alg{IdealMembership}(F,g)$
\begin{algorithmic}
\REQUIRE $I = \ideal{F}$, an ideal of $R$;\\
$g$, a polynomial in $R$.
\ENSURE $B = (g \in IR_0)$, a Boolean value.\\
\smallskip \hrule \smallskip

\STATE $e \gets \deg g^h$;
\STATE $d \gets 0$;
\LOOP
 \STATE $D_1 \gets D_0^d[I]$;
 \STATE $D_2 \gets D_0^d[I+\ideal{g}]$;
 \IF{$D_1 \neq D_2$}
  \STATE {\bf return} $false$;
 \ENDIF
 \STATE $C \gets g^h\cdot D_0^{d+e}[\ideal{F^h}]$;
 \IF{$h^d \in \LTg C$}
  \STATE {\bf return} $true$;
 \ENDIF
 \STATE $d \gets d+1$;
\ENDLOOP

\smallskip \hrule \smallskip
\end{algorithmic}
\end{algorithm}

Algorithm~\ref{alg:IdealMembership} fills in the gap left by the local ideal membership test proposed in Theorem~$4.6$ of~\cite{Leykin:NPD}, which missed the necessary assumption of homogeneity. 

\section{Numerical Primary Decomposition}\label{Sec:NPD}

There is a handful of methods for {\em symbolic} primary decomposition with implementations carried out for decomposition over $\bQ$. For a good overview see~\cite{Decker-Greuel-Pfister}.
  
A method for {\em numerical primary decomposition} (NPD) was introduced in~\cite{Leykin:NPD} and is intended to compute an {\em absolute} primary decomposition, i.e., decomposition over $\bC$. Conceptually it is set up in a framework where one can't use the symbolic techniques such as Gr\"obner bases and characteristic sets. 

\medskip

The following construction, inspired by the higher-order deflation \cite{LVZ-higher}, computes a superset of the primary components of an ideal. Consider an ideal $I=(f_1,\ldots,f_N)\subset R=\bC[x]$. Let $q = \sum_{|\beta|\leq d} a_\beta \p^\beta \in \bC[a][\p]$ be a linear differential operator of order at most $d$ with coefficients in the polynomial ring $\bC[a]$. Note there is a natural action of $\bC[a][\p]$ on $\bC[a][x]$.

The ideal generated by $f_1,\ldots,f_N$ and $q (x^\alpha f_i)$ for all $|\alpha|\leq d-1$ and $i=1,\ldots,N$ is called the {\em deflation ideal} of $I$ of order $d$ and denoted by~$\dId$.

We also refer to the {\em deflated variety} of order $d$,  $$\dXd = \bV(\dId) \subset \bC^{B(n,d)},$$
where $B(n,d)= n+\binom{n+d-1}{d}$ is the number of variables in $\bC[x,a]$.

The deflation ideal $\dId$ and, therefore, the deflated variety $\dXd$ does not depend on the choice of generators of the ideal $I$ (see~\cite[Proposition 2.7]{Leykin:NPD}).

Denote by $\pi_d : \dXd \to X$ the restriction of the natural projection from $\bC^{B(n,d)}$ to $\bC^n$. Note that this map is a surjection onto $X = X^{(0)} = \bV(I)$.

\begin{remark} For every point $x \in \bC^n$ the fiber of $\pi_d$ is isomorphic to the truncated dual space of order $d$, i.e., $$\pi_d^{-1}(x) \simeq D_x^{d}(I).$$
\end{remark}

The following statement enables us to compute all (including embedded) components associated to $I$.
\begin{theorem}[Theorem~3.8~of~\cite{Leykin:NPD}]\label{thm: visible deflation} Every component is {\em visible} at
some order $d$, i.e., for every prime $P\in \Ass(R/I)$, there exists $d$ such that the preimage  $\dYd = \pi_d^{-1}(Y)$ of the variety $Y=\bV(P)$ is an irreducible (isolated) component of the variety $\dXd = \bV(\dId)$.
\end{theorem}

The term ``visible'' reflects the tool that is used to ``see'' components: {\em numerical irreducible decomposition} (NID) algorithms such as in~\cite{SVW1}, which can detect isolated components numerically.

We call an isolated component $\dYd$ of $\dXd$ a {\em pseudocomponent} if $\pi_d(\dYd)$ is not a component of $X$.
We call pseudocomponents and embedded components of $X$ collectively {\em suspect components}.

Here is an outline of Algorithm~5.3 of~\cite{Leykin:NPD} that computes a superset of all associated components.
\begin{algorithm}\label{alg:NPD} $\calN = \operatorname{NPD}(I)$
\begin{algorithmic}
\REQUIRE $I$, ideal of $R$.
\ENSURE $\calN$, components associated to $I$.
\smallskip \hrule \smallskip

\STATE $\calN \gets \emptyset$
\STATE $d \gets 0$
\REPEAT
 \STATE $C_1 \gets$ isolated components of $\dId$ computed with an $\operatorname{NID}$ algorithm
 \STATE $C_2 \gets \big\{ Y\in C_1 \,|\, \pi_d(Y)\neq Z \mbox{ for all } Z \in \calN \big\}$
 \FORALL {$Y \in C_2$}
   \IF{ $Y$ is not a {\em pseudocomponent}}
      \STATE $\calN \gets \calN \cup \{ Y \}$
   \ENDIF
 \ENDFOR
 \STATE $d = d+1$;
\UNTIL{ a stopping criterion holds for $d$ }

\smallskip \hrule \smallskip
\end{algorithmic}
\end{algorithm}

There are two parts of the algorithm that need clarification:
a routine to determine whether a subvariety of $X$ is a pseudocomponent (the main topic of this article) and
a stopping criterion. A stopping criterion can be provided by bounding $d$ by the regularity index of the (global) Hilbert function.  The {\em a priori} bound is doubly exponential in the number of variables and, while demonstrating termination, is not practical. 


\begin{remark}
Isosingular decomposition \cite{Hauenstein-Wampler:isosingular} can also be used as a source of suspect components, although it is not known whether the procedure one may derive from the isosingular decomposition recovers all embedded components. 

Another way to produce suspect components is via iterated first-order deflation: consider the projections of the visible components of $X^{(1)}$, $(X^{(1)})^{(1)}$, etc.
\end{remark}

\section{Algorithms to detect embedded components}\label{Sec:embedded-test}

The problem of distinguishing embedded components from pseudocomponents can be condensed to the following.

\begin{problem}\label{problem:embedded-test}
Consider an ideal $I\subset R$ and a prime ideal $P\supset I$. Let $Q_1,\ldots,Q_r\supset I$ be the primary ideals in a primary decomposition of $I$ such that $\sqrt{Q_i}\subsetneq P$.

Given generators of $I$ and generic points $y_0\in\bV(P)$ and $y_i\in \bV(Q_i)$ ($i=1,\cdots,r$),
determine whether $P$ is an associated prime of $R/I$.

\smallskip

Equivalently, let $y_0=0\in \Var(P)$ be a generic point (by changing coordinates we may assume the origin is a generic point without loss of generality), determine whether
\begin{equation}\label{equ:no-pseudo}
IR_0 = Q_1R_0 \cap \cdots \cap Q_r R_0.
\end{equation}
\end{problem}

We describe an algorithm for when the suspect component $P$ is zero-dimensional, and then finally extend it to the fully general case.

\subsection{Suspect component of dimension $0$}\label{sec:0-dim-suspect}
Suppose the suspect component is of dimension $0$. Without loss of generality we may assume that it is the origin by a change of coordinates and also that $I = IR_0 \cap R$ because we may ignore components away from the origin.  We again let $\fm = \ideal{x_1,\ldots,x_N}$, the maximal ideal at the origin.  To simplify our notation, let $I = Q_0\cap J$ where $J = Q_1\cap \cdots \cap Q_r$ (as in Problem~\ref{problem:embedded-test}) and either
\begin{itemize}
  \item $Q_0 = R$, i.e., $V_0$ is a pseudocomponent; or
  \item $Q_0$ is a primary ideal with $\sqrt{Q_0}=\fm \in\Ass(R/I)$ and $Q_0$ does not contain $J = Q_1\cap \ldots \cap Q_r$, i.e., $V_0$ is a (true) component. 
\end{itemize}
The goal is to distinguish the two cases above.  Is $I = J$ or not?

We have $J = (I:\fm^{\infty})$.  For a generic linear form $\ell$ (so $\ell \notin \sqrt{I}$), the ideal $\ideal{\ell}$ is contained in $\sqrt{Q_0}$ but not in $\sqrt{Q_1},\ldots,\sqrt{Q_r}$, so then $J = (I:\ideal{\ell}^\infty )$.  This gives inclusions
 \[ I \subseteq (I:\ideal{\ell}) \subseteq J \]
with equality at the first inclusion if and only if there is no embedded component of $I$ at the origin.  Our general strategy will be to compute information about $I:\ideal{\ell}$ and $J$ and compare to $I$ in order to certify either that $I = I:\ideal{\ell}$ in which case there is no embedded component, or that $I \neq J$ in which case there is.

A major stumbling block is that we cannot get our hands directly on $I:\ideal{\ell}$ or $J$, or even on their truncated dual spaces.  In the former case, as discussed in Section \ref{Sec:eliminating}, we can compute $S_d := \ell\cdot D_0^{d+1}[I]$ which is a subspace of $D_0^d[I:\ideal{\ell}]$.  If for large enough $d$, $S_d$ contains all s-corners of $D_0[I]$, then we conclude that $D_0[I:\ideal{\ell}] = D_0[I]$, certifying that the origin is not embedded, but we cannot use this test to certify the origin is embedded.  On the other side, we compute subspaces $J_d := J \cap R_d$ of $J$, where $R_d$ denotes the space of polynomials with all terms of degree $\leq d$.  If $J_d \not\subset I$ for some $d$ then this certifies that the origin is embedded.  Similarly as $J_d$ is only a subset of $J$, we cannot use it to certify the origin is a pseudocomponent.  Both procedures are simultaneously iterated over $d$ until one terminates.

This algorithm is below, with the procedure $\alg{IdealTruncation}$ to compute $J_d$ defined later as Algorithm \ref{alg:IdealTruncation}. 
To find $\LT I$ (in particular, the s-corners of the staircase) we use the algorithm of~\cite{Krone:dual-bases-for-pos-dim}; see Remark~\ref{rem:HF-algo}.

\begin{algorithm}\label{alg:embedded-test-dim-0} $B = \operatorname{IsOriginEmbedded}(I)$
\begin{algorithmic}[1]
\REQUIRE $I = \ideal{F}$, an ideal of $R$
\ENSURE $B = \text{``origin is an embedded component of $I$''}$, a boolean value.
\smallskip \hrule \smallskip

\STATE compute $\LT I$
\STATE $d \gets 0$
\STATE $\ell \gets$ a generic linear form
\LOOP
  \STATE $J_d \gets \alg{IdealTruncation}(F,d)$ \label{line:run-B}
  \IF{$\LT J_d \not\subset \LT I$} \label{line:polys up to d}
  \STATE {\bf return} true
  \ENDIF \label{line:end-run-B}
  \STATE $S_d\gets \ell \cdot D^{d+1}_0[I]$
  \IF{$\p^{\alpha}\in \LTg S_d$ for all s-corners $x^\alpha$ of $\LT I$}\label{line:s-corners}
  \STATE {\bf return} false
  \ENDIF
  \STATE $d \gets d+1$
\ENDLOOP

\smallskip \hrule \smallskip
\end{algorithmic}
\end{algorithm}

\medskip
\begin{proof}[Proof of correctness and termination]
If the condition in Line~\ref{line:polys up to d} holds then there is some $f \in J_d \subset J$ such that $f\notin I$. Hence $J \neq I$ which implies the origin is an embedded component.  Because $J = \bigcup_d J_d$,  if $J \neq I$ then there is large enough $d$ for which $J_d$ will provide such a certificate.

Suppose $I \neq J$ and let $M_I$ denote the set of standard monomials of $I$.  Because $I$ and $I:\ideal{\ell}$ differ only by a component at the origin, $(I:\ideal{\ell})/I$ has finite $\bC$ dimension, and so $M_I \setminus M_{I:\ideal{\ell}}$ is also finite.  $M_{I:\ideal{\ell}}$ is closed under division, so $M_I \setminus M_{I:\ideal{\ell}}$ contains a monomial which is maximal in $M_I$, which is an s-corner of $I$.  Therefore if the condition in Line~\ref{line:s-corners} holds then $I=I:\ideal{\ell}$.  Because $D_0[I:\ideal{\ell}] = \bigcup_d S_d$, if $I = I:\ideal{\ell}$ then there is large enough $d$ for which $S_d$ will provide such a certificate.
\end{proof}

The staircases of $J$ and $I:\ideal{\ell}$ sit ``below'' the staircase of $I$.  Since $J_d$ is a subset of $J$, it provides an upper bound on the staircase of $J$, which can bound it away from $I$, proving that $J \neq I$.  On the other hand, since $S_d$ is a subset of $D_0[I:\ideal{\ell}]$, it provides a lower bound on the staircase of $I:\ideal{\ell}$.  If it includes the s-corners of $I$, then the staircases must agree.  See Figure \ref{fig:I-J-Jd}.

\begin{figure}[ht]
  \centering
  \includegraphics[width=.5\columnwidth]{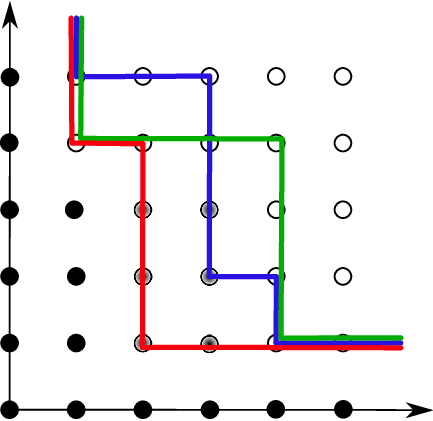}
  \caption[$I,J,J_d$]{Both $\Green{I}$ and $\Blue{\ideal{J_d}}$ are contained in $\Red{J}$. In general, no other containments hold. For $d\gg 0$, $\ideal{J_d}=J$.
  
  The set $\LT J \setminus \LT I$ of monomials is finite.}
  \label{fig:I-J-Jd}
\end{figure}

\subsubsection{Ideal truncation algorithm}

To complete Algorithm \ref{alg:embedded-test-dim-0} it remains to produce an algorithm for ideal truncations.

\begin{problem}[Local Interpolation]\label{problem:LocalInterpolation}
Let $d>0$ and $J = Q_1 \cap \cdots \cap  Q_r$  with each $Q_i$ a primary ideal such that each $V_i = \bV(Q_i)$ contains the origin (equivalently $J = JR_0 \cap R$).  Compute $J_{d} = J \cap R_d$.

We assume access to oracle $\calO_J$ which can sample random generic points $x$ on any $V_i$, and for any such $x$ and any $e \geq 0$ can compute $D_x^e[J]$.
\end{problem}

\begin{remark}
 We can use the tools of NPD to sample points on the suspect components of $I = J\cap Q_0$, which in particular means generic points on $\bV(Q_i)$ can be produced.  We can also compute truncated dual spaces $D_x^e[I]$ using the generators of $I$.  The local properties of $J$ and $I$ agree away from the origin and the origin is not a primary component of $J$.  Therefore simply by excluding the origin from consideration, we have access to the tools promised by $\calO_J$ and our oracle assumption is justified.
\end{remark}

To solve Problem \ref{problem:LocalInterpolation} we will use a form of interpolation.  We will sample generic points $x$ on the components of $J$, and compute dual spaces $D_x^e[J]$, which provide certain linear constraints on the evaluation and derivatives of polynomials $f \in JR_x$.  Finally we require a check to know when we have enough constraints to exactly define $J_d$.

For the general case, we first consider the {\em double truncations} of $J$:
\begin{equation}\label{eq:I-d-e}
J_d^e = \{ f\in R_d ~|~ \mbox{for all } i,\ D_{x}^e[Q_i] f = 0
\mbox{ for any generic point } x\in V_i \}.
\end{equation}
The following is a probabilistic algorithm to compute $J_d^e$ whenever we have a procedure to compute $D_x^e[J]$ for any generic point $x \in Q_i$ and any $e$.  In our case we have access to such a procedure because for any point $x$ away from the origin $D_x^e[J] = D_x^e[I]$.  Note $D_x^e[I]$ can be computed by the usual methods since the generators of $I$ are known.

\begin{algorithm}\label{alg:TruncatedTruncation}
$J_d^e = \TruncatedTruncation(\calO_J,d,e)$
\begin{algorithmic}
\REQUIRE $\calO_J$ an oracle as in Problem~\ref{problem:LocalInterpolation};\\
$d,e\in\bN$.\\
\ENSURE $J_d^e$ is as defined in (\ref{eq:I-d-e})
\smallskip \hrule \smallskip
\STATE $K \gets R_d$
\REPEAT
\STATE $oldK \gets K$
\STATE with $\calO_J$ choose generic points $x_i \in V_i$ for $i = 1,\ldots,r$.
\STATE $K \gets K \cap (D_{x_1}^e[J])^\perp \cap \cdots \cap (D_{x_r}^e[J])^\perp$
\UNTIL{$oldK = K$}
\STATE {\bf return} $J_d^e = K$
\smallskip \hrule \smallskip
\end{algorithmic}
\end{algorithm}
\begin{proof}[Proof of correctness and termination]
 Note that at every step $K \supseteq J_d^e$.  Suppose at some step that $K \neq J_d^e$.  There is $f \in K$ such that for some $V_i$ and any generic point $x \in V_i$, $f$ is not orthogonal to $D_x^e[J]$ by the definition of $J_d^e$.  The point $x_i$ chosen on $V_i$ is chosen generically, so the new value of $K$ is strictly contained in $oldK$.  Therefore when $K$ stabilizes, it must be equal to $J_d^e$.  Since $K$ is finite dimensional at every step, termination is guaranteed.
\end{proof}

\begin{proposition}\label{prop:truncation-stab}
 For any $d$, the chain
 \[ J_d^0 \supseteq J_d^{1} \supseteq J_d^2 \supseteq \cdots \]
 stabilizes to $J_d$.  That is, $J_d^e = J_d$ for all $e$ sufficiently large.
\end{proposition}
\begin{proof}
 For any point $x$ recall from Remark \ref{rmk:local-dual} that polynomial $f$ has $p(f) = 0$ for all $p \in D_{x}[I]$ if and only if $f \in IR_{x} \cap R$, and note that $IR_{x} \cap R = \bigcap_{x \in V_i} Q_i$.  Choosing a point $x_i$ from each $V_i$, the set $\bigcup_e J_d^e$ is the set of polynomials $f \in R_d$ orthogonal to each dual space $D_{x_i}[I]$.  Because every $V_i$ contains at least one of the points $x_1,\ldots,x_r$,
  \[ \bigcap_i (D_{x_i}[I])^\perp = Q_1 \cap \cdots \cap Q_r = J. \]
 Therefore $\bigcup_e J_d^e = J_d$.  Since $J_d$ has finite $\bC$-dimension, there must be some $e$ at which stabilization occurs.
\end{proof}

This fact suggests an algorithm for computing $J_d$ from the double truncations, in particular for each value of $e \geq 0$ compute $J_d^e$ until some $J_d^e \subseteq J$.  A naive stopping criterion for this procedure might be when $J_d^e = J_d^{e+1}$ for some $e$, but this will not work as the following example illustrates.

\begin{example}\label{exa:double-truncation-non-containment}
Let $I = \ideal{x^k+y, y^k} \subset R = \bC[x,y,z]$, a positive-dimensional primary ideal. The reader may check that
\begin{align*}
I_1^1 &= y\\
I_1^2 &= y\\
&...\\
I_1^k &= I_1 = 0
\end{align*}
This example shows that equality of two subsequent $I_d^e$ and $I_d^{e+1}$ is not a valid stopping criterion. Also, note that $I_1^e \not\subset I$ for $e<k$.
\end{example}

Instead we require a method to check if $J_d^e \subseteq J$.  First note that for any $\bC$-vector subspaces $V$ and $W$ with $V$ finite dimensional, a generic vector $v \in V$ is in $W$ if and only if $V \subseteq W$.  Therefore it is sufficient for our purposes to check if a randomly chosen polynomial $g \in J_d^e$ is contained in $J$.  Such a membership test was described in Algorithm \ref{alg:IdealMembership} when generators for the ideal were known, but in this case we do not know generators of $J$, only for $I$, so the algorithm must be modified.

\begin{proposition}\label{prop:colon-dim0} Let $I \subseteq R$ be an ideal and $J = (I:\fm^{\infty})$.  A polynomial $g \in R$ is in $J$ if and only if $\sqrt{I : \ideal{g}} = \fm$.
 \end{proposition}

 \begin{proof}
  Let $I=Q_0\cap Q_1 \cap \cdots \cap Q_r$ be a primary decomposition with $\sqrt{Q_i)} \neq \fm$ for $i >0$ and $\dim Q_0 = 0$ or $Q_0 = R$.  Let $J = (I:\fm^{\infty}) = Q_1 \cap \cdots \cap Q_r$.
 
  If $g \notin J$, then $g \notin Q_i$ for some $i > 0$, so $I : g \subset P_i$ where $P_i$ is the prime associated to $Q_i$.  Since $P_i$ has positive dimension, so does $I:\ideal{g}$.  Conversely if $I: \ideal{g}$ is positive-dimensional, it is contained in some positive-dimensional prime $P$.  Then $I$ has a primary component $Q_i$ with $Q_i \subset P$ and $g \notin Q_i$.  Since $Q_i \subset P$, it has positive dimension so $g \notin J$.
 \end{proof}

 To check that this condition holds we use the dual space of $\ideal{F^h}:\ideal{g^h}$, where $I = \ideal{F}$, to find g-corners of $I:\ideal{g}$, just as in Algorithm \ref{alg:IdealMembership}.  $I:\ideal{g}$ is zero-dimensional if and only if for every variable $x_i$ there is a g-corner of $I:\ideal{g}$ of the form $x_i^a$.
 
 The algorithm we present searches for g-corners of the form $x_i^a$ iterating over the degree, but cannot prove the non-existence of such g-corners.  As a result, our algorithm to determine if $g \in J$ will stop at some cutoff degree $c$, return true if it can certify that $g \in J$, and return false if the cutoff value is reached.

\begin{algorithm}\label{alg:CheckCandidate} $B = \IsWitnessPolynomial(F,g,c)$
\begin{algorithmic}
\REQUIRE $I = \ideal{F}$, an ideal of $R$;\\
$g$, a polynomial in $R$;\\
$c$, a degree cutoff.
\ENSURE $B = \text{$true$ iff $g \in (I:\fm^\infty)$ and $c$ sufficiently large}$.\\
\smallskip \hrule \smallskip

\STATE $e \gets \deg g^h$
\STATE $d \gets 0$
\STATE $G \gets \{\}$ (the g-corners of $I:\ideal{g}$)
\REPEAT
 \STATE $C \gets$ new g-corners of $I:\ideal{g}$ computed from $g^h\cdot D_0^{d+e}[\ideal{F^h}]$
 \STATE append $C$ to $G$
 \IF{$x_i^{a_i} \in G$ for all $i = 1,\ldots ,n$ and any $a_i$}
  \STATE {\bf return} $true$
 \ENDIF
 \STATE $d \gets d+1$
\UNTIL{$d > c$}
\STATE {\bf return} $false$

\smallskip \hrule \smallskip
\end{algorithmic}
\end{algorithm}

Equipped with this algorithm for checking if a polynomial $g$ is in $J$, and the double truncation algorithm above, we can now compute $J_d$ as follows.

\begin{algorithm}\label{alg:IdealTruncation}
$J_d = \alg{IdealTruncation}(F,d)$
\begin{algorithmic}
\REQUIRE $I = \ideal{F}$, an ideal of $R$;\\
$d\in\bN$.\\
\ENSURE $J_d = (I:\fm^{\infty}) \cap R_d$
\smallskip \hrule \smallskip
\STATE $e\gets 0$
\LOOP
\STATE $J_d^e \gets \TruncatedTruncation(\calO_J,d,e)$
\STATE $g \gets $ random polynomial chosen from $J_d^e$
  \IF{$\alg{IsWitnessPolynomial}(F,g,e)$}
  \STATE {\bf return} $J_d = J_d^e$
  \ENDIF
\STATE $e \gets e + 1$
\ENDLOOP
\smallskip \hrule \smallskip
\end{algorithmic}
\end{algorithm}

\begin{proof}[Proof of correctness and termination]
 If $\alg{IsWitnessPolynomial}(F,g,e)$ returns true then $g$ must be in $J_d$.  By Proposition \ref{prop:truncation-stab} $J_d^e \supseteq J_d$, so randomly chosen $g$ from $J_d^e$ has $g \in J_d$ if and only if $J_d^e = J_d$ almost surely.  This proves correctness.
 
 To prove termination, first note that there is $e_0$ such that $J_d^e = J_d$ for all $e\geq e_0$ by Proposition \ref{prop:truncation-stab}.  It remains to show that $\alg{IsWitnessPolynomial}(F,g,e)$ will return true for some $e\geq e_0$.

 For any $g \in J_d$, let $c(g)$ denote the minimum cutoff value $c$ such that $\IsWitnessPolynomial(F,g,c)$ returns true.  Let $\{b_1,\ldots,b_s\}$ be a $\bC$-basis for $J_d$, so we can express $g \in J_d$ as $g = \sum_{i=1}^s a_ib_i$.  For any given value of $c$, the set of polynomials 
  \[ W_c = \{g \in J_{d} \mid c(g) = c \} \]
 can be described by a finite set of algebraic conditions on $a_1,\ldots,a_s$, so $W_c$ is a constructible set.  In particular, there is some $c_0$ such that $W_{c_0}$ is Zariski open, so $\IsWitnessPolynomial(F,g,c_0)$ will return true for generic $g \in J_{d}$.  For $e \geq \max\{e_0,c_0\}$, a generic polynomial $g$ sampled from $J_{d}^e$ will be in $J_d$, and $\IsWitnessPolynomial(F,g,e)$ will certify this fact.
\end{proof}

This completes Algorithm \ref{alg:embedded-test-dim-0} for determining if the origin is a zero-dimensional embedded component of ideal $I$.

\subsubsection{An example computation}
\begin{example}\label{example:cyclic}
Consider the cyclic 4-roots problem:
\begin{align*}
I &=  \big( x_1+x_2+x_3+x_4,\ x_1x_2+x_2x_3+x_3x_4+x_4x_1,\\
  & x_1x_2x_3+x_2x_3x_4+x_3x_4x_1+x_4x_1x_2,\ x_1x_2x_3x_4-1 \big )\,.
\end{align*}

Computing {\tt numericalIrreducibleDecomposition} of the first-order deflated variety $X^{(1)}=\Var(I^{(1)})$ 
we obtain witness sets representing isolated components of $X^{(1)}$ that project to 
\begin{itemize}
\item two irreducible curves, isolated components that are visible and can be discovered by  {\tt numericalIrreducibleDecomposition} of $X=\Var(I)$, and
\item eight points, approximations to $\{(a,b,-a,-b) \mid a\in\{\pm 1,\pm i\},\, b=\pm a \} $ which are {\em suspect} components.
\end{itemize}

For an approximation of the point $(i, -i, -i, i)$, {\tt isPointEmbedded} produces a witness polynomial,
\begin{verbatim}
  witness poly: (d',d) = (1, 4)
  (.586169+.361093*ii)*x_1+(.776351+.36685*ii)*x_2+
  (.586169+.361093*ii)*x_3+(.776351+.36685*ii)*x_4
\end{verbatim}
\noindent
showing that this point is an embedded component. Same conclusion holds for all suspect points. 

The associated primes (computed over $\bQ$ with a symbolic {\em Macaulay2}\, routine) are
\[
\begin{array}{rll}
  \Ass(R/I) =  \big\{ & (x_2+x_4,x_1+x_3,x_3x_4+1),\\
                    & (x_2+x_4,x_1+x_3,x_3x_4-1),\\
                    & (x_4-1,x_3+1,x_2+1,x_1-1),\\
                    & (x_4-1,x_3-1,x_2+1,x_1+1),\\
                    & (x_4+1,x_3+1,x_2-1,x_1-1),\\
                    & (x_4+1,x_3-1,x_2-1,x_1+1),\\
                    & (x_3+x_4,x_2+x_4,x_1-x_4,x_4^2+1),\\
                    & (x_3-x_4,x_2+x_4,x_1+x_4,x_4^2+1)& \big\}
\end{array}
\]
confirming the numerical results.
\end{example}

\subsection{Suspect component of positive dimension}\label{subsec:positive-dim-suspect}

Let $P_0$ be the vanishing (prime) ideal of suspect component $V_0$; let $d_0=\dim V_0 >0$.

We would like to deduce and rely on a Bertini-type theorem (Theorem~\ref{thm:Ass-of-a-generic-slice}) that, roughly, says that given an ideal $I\subset R$ with $\min_{P\in\Ass(R/I)} \dim P \geq d_0$ we have a correspondence between $\Ass(R/I)$ and $\Ass(R/(I+L))$ where $L$ is a generic affine plane of codimension $d_0$. This correspondence is one-to-one for components of dimension $d_0+1$; there could be multiple 0-dimensional components in $\Ass(R/(I+L))$ ``witnessing'' components of dimension $d_0$ in $\Ass(R/I)$.

\begin{lemma}\label{lem:colon-hyperplane-commute}
Let $I$ be an ideal and $f$ be an element of $R$. Then for a generic (affine) linear function $h\in R$
\[(I+H):F=(I:F)+H,\text{ where }F = \ideal{f}, H = \ideal{h}.\]
\end{lemma}
\begin{proof}{\small
(The proof follows closely the argument at {\tt mathoverflow.net/questions/143076} given by Hailong Dao.)}

If $I+F=R$ then $I:F = I$ and $(I+H):F=I+H$;  therefore, assume $I+F\neq R$. 
The set of associated primes $A=\Ass(R/(I+F))$ is finite, hence, a generic $h$ would be a non-zerodivisor on $R/(I+F)$. To see that it is enough to notice that the set of zerodivisors is exactly $\bigcup_{P\in A}P$ and that $n+1$ generic linear functions generate $R$.  

Consider the exact sequence
\[ 
0 \to R/(I:F) \to R/I \to R/(I+F) \to 0
\]
with first map being the multiplication by $f$. Tensoring with $R/H$ we get another exact sequence,
\[
0 \to R/(I:F+H) \to R/(I+H) \to R/(I+F+H) \to 0,
\]
coming from a long exact sequence for $\Tor^R(\cdot,R/H)$ and the fact that $\Tor_1^R(R/(I+F),R/H)=0$ as $H$ is a non-zerodivisor on $R/(I+H)$.

On the other hand, the first exact sequence with $I$ replaced by $I+H$
says that the leftmost term in the second sequence should be
isomorphic to $R/((I+H):F)$, which proves the Lemma.
\end{proof}

\begin{lemma}\label{lem:generic-hyperplane-section-no-embedded}
In the notation of the previous proposition, if $I$ defines a scheme with no embedded components, then so does $I+H$ for a generic $H$. 
\end{lemma}
\begin{proof} See \cite[Example 3.4.2(6)]{Flenner-O-Carroll-Vogel}: the condition of ``having no embedded components'' satisfies the Generic Principle~\cite[Theorem 3.3.10]{Flenner-O-Carroll-Vogel}.  \end{proof}

\begin{lemma}\label{lem:no-extraneous-ass-primes}
Let $I = Q_1 \cap ... \cap Q_r$ be a primary decomposition. Then for a generic hyperplane $H$ the natural injection
$R/I \hookrightarrow \bigoplus_i (R/Q_i)$ induces an injection 
\[
R/(I+H) \hookrightarrow \bigoplus_i (R/(Q_i+H)).
\]
In particular, $\Ass(R/(I+H)) \subset \{P+H\mid P\in\Ass(R/I)\}$.
\end{lemma}
\begin{proof}
Consider the short exact sequence
\[
0 \to R/I \to \bigoplus_i (R/Q_i) \to C \to 0.
\]
As in the proof of Lemma~\ref{lem:colon-hyperplane-commute} we see that $\Tor_1(C,R/H)=0$ for a generic hyperplane $H$.
Indeed, this follows from a generic $H$ being a non-zerodivisor due to the finiteness of $\Ass C$. 
\end{proof}

\begin{theorem}\label{thm:Ass-of-a-generic-slice} 
Let $I$ be an ideal of $R=\bC[x_1,\ldots,x_n]$ and let $L$ be the vanishing ideal for a generic affine $(n-k)$-plane. Then 
\begin{align*}
\Ass(R/I+L) &= \{P+L \mid P\in \Ass(R/I),\ \dim(P)>k\}\  \cup \\
&\bigcup_{
\substack{P\in\Ass(R/I) \\ \dim(P)=k}
} \Ass(R/(P+L))\,.
\end{align*}
\end{theorem}
\begin{proof}
Lemma \ref{lem:generic-hyperplane-section-no-embedded} says, in particular, that for a primary ideal $Q$ the ideal $Q+L$ has no embedded components; therefore, $Q+L$ is either primary or 0-dimensional (in case $\dim(Q)=\codim(L)$). 

Now, on one hand, Lemma \ref{lem:no-extraneous-ass-primes} says that $I+L$ has no extraneous associated primes: all components have to come from $Q+L$ where $Q$ is an ideal in a primary decomposition of $I$. On the other hand, Lemma~\ref{lem:colon-hyperplane-commute} implies that every $P\in\Ass(R/I)$ is witnessed by $\Ass(R/(P+L))$, since one can arrange an $f\in R$ so that $\Ass(R/(I:f)) = \{P\}$.

Finally, $\Ass(R/(P+L))$ contains one element $P+L$ when $\dim(P)>k$, is empty when $\dim(P)<k$, and is a finite set of maximal ideals when $\dim(P)=k$.  
\end{proof}
Using this theorem we can reduce the case of a component of positive dimension to the embedded component test in the 0-dimensional case, i.e., the algorithms in previous subsections of this section. Indeed, for a suspect component $V$ of dimension $k$ one can intersect the scheme with a random affine plane $\Var(L)$ of codimension $k$ and ask whether a point of $V\cap \Var(L)$ is an embedded component of that intersection.

\begin{example}\label{example:5-lines-2-planes}
The radical ideal
\begin{align*}
I &=  \ideal{x,z} \cap \ideal{x^2-y^2,y+z} \cap \ideal{x^2-z^2,x+2y} \cap \ideal{(x-1)y}
\end{align*}
describes a union of 5 lines and 2 planes.

A {\em Macaulay2} script that takes a set of generators of $I$ proceeds to construct the first deflation ideal $I^{(1)}$ discovering 13 isolated components of $\Var(I^{(1)})$ that project to suspect components in $\bC^3$. Its summary reads
{\small
\begin{verbatim}
total: 13 suspect components
true components: {0, 3, 6, 9, 10, 11, 12}
\end{verbatim}
}
\noindent
displaying the correct list of 7 true components and correctly discarding all pseudocomponents. 

This example is built primarily to test various scenarios for pseudocomponents: there is a positive-dimensional pseudocomponent -- the intersection of two planes -- and several 0-dimensional pseudocomponents. For the former, Theorem~\ref{thm:Ass-of-a-generic-slice} is utilized to reduce to the 0-dimensional case.
One of the latter -- the origin -- has a non-empty set of s-corners, which engages non-trivially one of the termination modes of Algorithm~\ref{alg:embedded-test-dim-0}. Here is the corresponding excerpt:
{\small
\begin{verbatim}
                2
-- s-corners: {y z}
                              3   2      2   3   2           2   ...
-- LM(dual of colon ideal): {x , x y, x*y , y , x z, x*y*z, y z, ...
V(z, y, x), contained in 6 other components, is a PSEUDO-component
\end{verbatim}
}
\noindent
The output can be interpreted to say that $\p_y^2\p_z$ belongs to $\ell\cdot D_0^4[I]$, for a generic linear form $\ell$, hence the conclusion. 
\end{example}

\section{Numerical ingredients and Conclusion}\label{sec:oracles-and-conclusion}
One may think that all we compute could be computed by symbolic primary decomposition algorithms that employ Gr\"obner bases. Let us reiterate that we view our task in the framework of numerical AG which, on one hand, implicitly prohibits the use of polynomial rewriting techniques and, on the other hand, does not straightforwardly extend to the scheme-theoretical setting.  

We made a comment in the introduction saying that our algorithm can be viewed as symbolic if all parts of the input to the Main Problem are assumed exact. That may clarify understanding of the paper but, in reality, we have hybrid algorithms that rely conceptually on two numerical oracles:
\begin{enumerate}
\item[\bf O1] Given a polynomial system $F$, return an approximation to a {\em generic} (in practice, random) point on each irreducible component of $\Var(F)$ with any prescribed error bound. 

\item[\bf O2] Compute an approximate kernel of an approximate matrix given a threshold for the singular values. 
\end{enumerate}
For the theoretical purposes of this paper, these oracles are blackboxes; however, algorithms exist in practice for accomplishing the tasks of both of them.  Oracle {\bf O1} can be implemented using polynomial homotopy continuation techniques under the hood of numerial irreducible decomposition. Oracle {\bf O2} is needed for a Macaulay dual space computation and boils down to singular value decomposition techniques.  Numerical questions that may arise in connection to the oracles,  such as the question of numerical stability, are beyond the scope of this paper.

\begin{remark}
All algorithms in the paper are {\em consistent with respect to numerical error}\,: the algorithms produce discrete output (a Boolean value, finite sets of integers, etc.) and there exists $\varepsilon>0$ such that for all input with $|\text{error}|<\varepsilon$ the output is {\em the same}.  This is the {\em only} kind of numerical stability we want to mention.

For example, suppose the task is to recover the rank of the Jacobian $\frac{\p F}{\p x} (y)$ at a generic point $y$ on a component of $\Var(F)$. Then the oracle {\bf O1} provides an approximation $y_\varepsilon$ to $y$ with $|y_\varepsilon - y|<\varepsilon$ and a part of {\bf O2} recovers the {\em numerical rank} by counting singular values of $\frac{\p F}{\p x} (y_\varepsilon)$ above the threshold $\delta>0$. There exists $\delta$ and $\varepsilon$ such that the numerical rank is the same for any choice $y_\varepsilon$, in particular, for $y_\varepsilon=y$ and, therefore, coincides with the true rank.
\end{remark}

Note that the theoretical existence of a threshold (such as $\varepsilon$ in the the example of the remark above) is usually not backed up by an efficient algorithm. As a consequence, there is only a handful of scenarios in numerical AG where an approximate output can be validated or certified. Most numerical AG algorithms, including ours, should be perceived as heuristic.

\medskip

An ability to numerically ``see'' embedded components demonstrates a potential for extending the arsenal of numerical AG to build numerical descriptions of affine and projective complex schemes as well as numerical counterparts of the established Gr\"obner bases techniques in the general scheme-theoretic setting. While we show that ability conceptually by providing the first numerical algorithm for an embedded component test, we have no illusions about its efficiency: our current implementation does not scale far beyond the examples given here.

\bibliographystyle{plain}
\bibliography{bib}

\begin{thebibliography}{10}

\bibitem{Decker-Greuel-Pfister}
Wolfram Decker, Gert-Martin Greuel, and Gerhard Pfister.
\newblock Primary decomposition: algorithms and comparisons.
\newblock In {\em Algorithmic algebra and number theory (Heidelberg, 1997)},
  pages 187--220. Springer, Berlin, 1999.

\bibitem{Flenner-O-Carroll-Vogel}
H.~Flenner, L.~O'Carroll, and W.~Vogel.
\newblock {\em Joins and intersections}.
\newblock Springer Monographs in Mathematics. Springer-Verlag, Berlin, 1999.

\bibitem{M2www}
Daniel~R. Grayson and Michael~E. Stillman.
\newblock Macaulay2, a software system for research in algebraic geometry.
\newblock Available at \verb+www.math.uiuc.edu/Macaulay2/+.

\bibitem{Singular-book-02}
Gert-Martin Greuel and Gerhard Pfister.
\newblock {\em A {\bf {S}ingular} introduction to commutative algebra}.
\newblock Springer, Berlin, extended edition, 2008.
\newblock With contributions by Olaf Bachmann, Christoph Lossen and Hans
  Sch{\"o}nemann, With 1 CD-ROM (Windows, Macintosh and UNIX).

\bibitem{Hauenstein-Wampler:isosingular}
Jonathan~D. Hauenstein and Charles~W. Wampler.
\newblock Isosingular sets and deflation.
\newblock {\em Found. Comput. Math.}, 13(3):371--403, 2013.

\bibitem{Krone:dual-bases-for-pos-dim}
Robert Krone.
\newblock Numerical algorithms for dual bases of positive-dimensional ideals.
\newblock {\em Journal of Algebra and Its Applications}, 12(06):1350018, 2013.

\bibitem{Krone:NumericalHilbert}
Robert Krone.
\newblock {N}umerical{H}ilbert package for {M}acaulay2.
\newblock 2014.
\newblock Preprint available at \url{http://arxiv.org/abs/1405.5293}.

\bibitem{KL:eliminating-dual-spaces}
Robert Krone and Anton Leykin.
\newblock Eliminating dual spaces.
\newblock 2015.
\newblock To appear in Journal of Symbolic Computation.

\bibitem{NAGwww}
Anton Leykin.
\newblock {NumericalAlgebraicGeometry package for Macaulay2}.
\newblock Available at \verb+people.math.gatech.edu/~aleykin3/NAG4M2/+.

\bibitem{Leykin:NPD}
Anton Leykin.
\newblock Numerical primary decomposition.
\newblock In {\em International Symposium on Symbolic and Algebraic
  Computation}, pages 165--172. ACM, 2008.

\bibitem{Leykin:NAG}
Anton Leykin.
\newblock Numerical algebraic geometry.
\newblock {\em Journal of Software for Algebra and Geometry}, 3:5--10, 2011.

\bibitem{LVZ-higher}
Anton Leykin, Jan Verschelde, and Ailing Zhao.
\newblock Higher-order deflation for polynomial systems with isolated singular
  solutions.
\newblock In {\em Algorithms in algebraic geometry}, volume 146 of {\em IMA
  Vol. Math. Appl.}, pages 79--97. Springer, New York, 2008.

\bibitem{Mourrain:inverse-systems}
B.~Mourrain.
\newblock Isolated points, duality and residues.
\newblock {\em J. Pure Appl. Algebra}, 117/118:469--493, 1997.
\newblock Algorithms for algebra (Eindhoven, 1996).

\bibitem{SVW1}
A.J. Sommese, J.~Verschelde, and C.W. Wampler.
\newblock Numerical decomposition of the solution sets of polynomial systems
  into irreducible components.
\newblock {\em SIAM J.\ Numer.\ Anal.}, 38(6):2022--2046, 2001.

\bibitem{SVW9}
A.J. Sommese, J.~Verschelde, and C.W. Wampler.
\newblock Introduction to numerical algebraic geometry.
\newblock In A.~Dickenstein and I.~Emiris, editors, {\em Solving polynomial
  equations}, pages 301--338. Springer-Verlag, 2005.

\bibitem{Sommese-Wampler-book-05}
Andrew~J. Sommese and Charles~W. Wampler, II.
\newblock {\em The numerical solution of systems of polynomials}.
\newblock World Scientific Publishing Co. Pte. Ltd., Hackensack, NJ, 2005.

\end{thebibliography}

\end{document}